\documentclass{article}
\usepackage[inner=4cm, outer=2cm, top=3cm, bottom=2cm,]{geometry}
\usepackage{amsmath, amsfonts, amssymb, color, comment, longtable, amsthm, breqn, longtable, microtype, xfrac, comment, mathtools, multicol, clipboard, mleftright, thmtools}
\usepackage{scalerel}
\usepackage{stackengine,wasysym}
\makeatletter
\newcommand{\altqedhere}{%
  \ifmeasuring@\else\sbox0{\popQED}\fi
  \tag*{\qedsymbol}%
}
\makeatother
\newcommand\reallywidetilde[1]{\ThisStyle{%
  \setbox0=\hbox{$\SavedStyle#1$}%
  \stackengine{-.1\LMpt}{$\SavedStyle#1$}{%
    \stretchto{\scaleto{\SavedStyle\mkern.2mu\AC}{.4\wd0}}{.7\ht0}%
  }{O}{c}{F}{T}{S}%
}}
\makeatletter
\newcommand{\leqnomode}{\tagsleft@true}
\newcommand{\reqnomode}{\tagsleft@false}
\makeatother
\reqnomode
\newtheorem{theorem}{Theorem}[section]
\AfterEndEnvironment{theorem}{\noindent\ignorespaces}
\newtheorem{lemma}[theorem]{Lemma}
\newtheorem{proposition}[theorem]{Proposition}
\theoremstyle{definition}
\newtheorem{definition}[theorem]{Definition}
\AfterEndEnvironment{definition}{\noindent\ignorespaces}
\newtheorem{remark}[theorem]{Remark}

\usepackage[style=alphabetic]{biblatex}
\newcommand\numberthis{\addtocounter{equation}{1}\tag{\theequation}}

\usepackage{hyperref}
\allowdisplaybreaks
\begin{document}
\title{Explicit height bounds on modular polynomials for the elliptic $j$-invariant, cube root of $j$, and Weber modular function $\mathfrak f$}
\author{Abraham \textsc{Zhang}}
\maketitle
\begin{abstract}
\noindent We improve an explicit upper bound on the height of the modular polynomial $\Phi_N$ given by Breuer, G\'omez, and Pazuki. We also prove new explicit upper bounds on the height of modular polynomials for the cube root of $j$ and the Weber modular function $\mathfrak f$.
\end{abstract}
\section{Introduction}
The classical modular polynomials $\Phi_N(X,Y)\in\mathbb Z[X,Y]$ are an important tool in modern cryptography. For example, they are used in isogeny based cryptography and efficient point-counting on elliptic curves (needed for elliptic curve cryptography). This is due to $\Phi_N$ vanishing at pairs of $j$-invariants of elliptic curves linked by cyclic $N$-isogeny.\par
$\Phi_N$ have rapidly growing (with $N$) and large coefficients. To compute the $\Phi_N$ one needs good bounds on the size of their coefficients. So we introduce the following definition:
\begin{definition}The logarithmic height of a complex polynomial $P$ is $h(P):=\allowbreak\ln\max_{c\text{ coefficient of }P}|c|$.
\end{definition}Another reason for computing height bounds is that with them we can estimate how large of an $N$ we can handle when working with $\Phi_N$.\par
In this paper we focus on computing height bounds of the modular polynomials $\Phi_{\gamma_2,N}$ and $\Phi_{\mathfrak f,N}$. Here $\gamma_2$ is the cube root of the $j$-invariant and $\mathfrak f$ is the Weber modular function. These polynomials have the advantage over the classical case in that their coefficients are sparser and smaller. The polynomial $\Phi_{\gamma_2,N}$ ($\Phi_{\mathfrak f, N}$ respectively) is sparser than $\Phi_N$ by a factor of about 3 (24 respectively) \cite[Section 7]{BrSu102}. For the heights, \cite{BGP23} proves
\begin{theorem}\label{old}For all integers $N\ge1$, the height of the modular polynomial $\Phi_N$ is bounded by
$$h(\Phi_N)\le 6\psi(N)(\ln N-2\lambda_N+9.5387).$$
\end{theorem}
In this paper we prove the following 4 related theorems:
\begin{theorem}\label{gamma}For all integers $N\ge2$ coprime to $3$, the height of the modular polynomial $\Phi_{\gamma_2,N}$ is bounded by
$$
h ( \Phi_{\gamma_2,N}) <2 \psi( N ) ( \ln N-2 \lambda_{N}+\ln \ln N + 4.54).
$$
\end{theorem}
\def\gammanew{
\begin{theorem}\label{gammanew}
For all integers $N\ge1$ coprime to $3$, the height of the modular polynomial $\Phi_{\gamma_2,N}$ is bounded by $$h(\Phi_{\gamma_2,N})<2\psi(N)(\ln N-2\lambda_N+9.06).$$
\end{theorem}}
\gammanew
\begin{theorem}\label{f}For all integers $N\ge2$ coprime to $48$, the height of the modular polynomial $\Phi_{\mathfrak f,N}$ is bounded by
$$
h ( \Phi_{\mathfrak f,N}) < \frac{\psi( N )}{11.98} ( \ln N-2 \lambda_{N}+\ln \ln N + 67.6).
$$
\end{theorem}
\def\improve{
\begin{theorem}\label{improve}For all integers $N\ge1$, the height of the modular polynomial $\Phi_{N}$ is bounded by $$h(\Phi_N)\le 6\psi(N)(\ln N-2\lambda_N+7.6928).$$
\end{theorem}}
\def\fnew{
\begin{theorem}\label{fnew}For all integers $N\ge1$ coprime to $48$, the height of the modular polynomial $\Phi_{\mathfrak f,N}$ is bounded by$$h\left(\Phi_{\mathfrak f,N}\right)<\frac{\psi(N)}{12}(\ln N-2\lambda_N+81.5).$$
\end{theorem}}
\fnew\par
For the two theorems with $\ln\ln N$, we follow the approach of \cite{BrPa22}. Note that the theorems with $\ln\ln N$ are better for reasonable $N$ (for example, for $N<\exp(\exp(4))<10^{22}$). We prove the theorems without $\ln\ln N$ because they have the right asympototic, which we also prove.\par
For the two theorems without $\ln\ln N$, we follow the approach of \cite[Section 3]{BGP23}, with an improvement resulting from estimating sums involving the Euler totient function $\varphi$. This also results in an improvement to the constant in \autoref{old}, namely
\improve
We also prove that the bounds are close to asymptotically correct.\par
Since $|\mathfrak f|$ is invariant only for a strict subgroup of $\mathrm{PSL}_2(\mathbb Z)$, the 2 theorems involving $\mathfrak f$ require theory about Hecke eigenforms modular for subgroups of $\mathrm{PSL}_2(\mathbb Z)$. In particular we use results from \cite{KiS25}. \cite{AS10} also has related results. We could only prove our bounds for a denominator of $12-\varepsilon$.
\section{Preliminaries}
Let $\ell$ and $N$ be positive integers. $N$ will be the order of our modular polynomials for a modular function of level $\ell$, so we assume $N$ is coprime to $\ell$. Let $$C_{\ell,N}:=\left\{\begin{pmatrix}a&\ell b\\0&d\end{pmatrix}\in\mathrm{M}_2(\mathbb Z)\middle|
ad=N,0\le b<d,\gcd(a,b,d)=1\right\}.$$
Elements of $C_N:=C_{1,N}$ correspond to cyclic $N$-isogenies from a fixed complex elliptic curve. $C_{\ell,N}$ has cardinality given by the Dedekind psi function
$$\psi(N)=N\prod_{\substack{p\text{ prime}}}\left(1+\frac1p\right).$$\par
Let $\tau$ be a point on the Poincar\'e upper half plane $\mathbb H$. Let $q_{\ell}:=e^{\frac{2\pi i\tau}{\ell}}$. For any function $f:A\to(B,\cdot)$ let $f^{ n}(\tau):=(f(\tau))^n$. The $j$-invariant of an elliptic curve $\frac{\mathbb C}{\mathbb Z+\tau\mathbb Z}$ is given by $$j(\tau)=1728\frac{\left(1+240\sum_{t=1}^{\infty}\frac{t^3q^t}{1-q^t}\right)^3}{\eta^{24}(\tau)},$$
where $$\eta(\tau)=q_{24}\prod_{n=1}^{\infty}(1-q^n)$$ is the Dedekind eta function.
\begin{definition}
The level of a modular function $f$ is the level of the lowest level congruence subgroup for which $f$ is a modular function for.\end{definition}
\begin{definition}   
Let $f$ be a modular function of level $\ell$. Let $N$ be coprime to $\ell$. Let $f_N(\tau):=f(N\tau)$. The minimal polynomial $\Phi_{f,N}(X,f)$ of $f_N$ over $\mathbb C(f)$  is called the modular polynomial of order $N$ for $f$. The polynomial $\Phi_N=\Phi_{j,N}$ is called the classical modular polynomial.
\end{definition}
Later on we will estimate the height of $\Phi_{f,N}$ by its Mahler measure. To do so we will need the form given by the following proposition:
\begin{proposition}\label{product prop}
Suppose $f$ is a modular function of level $\ell$ and $N$ coprime to $\ell$. Suppose $f$ is a Hauptomodul for a level $\ell$ congruence subgroup $\Gamma$ of $\mathrm{SL}_2(\mathbb Z)$ such that $\Gamma$ acts on the right of $\{\Gamma\sigma\}_{\sigma\in C_{\ell,N}}$ by multiplication. Then
\begin{align}\Phi_{f,N}(X,f)=\prod_{\sigma\in C_{\ell,N}}(X-f\circ\sigma).\label{product}\end{align}
Moreover, if the only poles of $f$ are at $\Gamma\infty$, then the coefficients of $\Phi_{f,N}(X,f)$ are polynomials in $f$.
\end{proposition}
\begin{proof} We first prove that the $\Gamma\sigma$'s are distinct. Suppose $\sigma_1,\sigma_2\in C_{\ell,N}$ represent the same coset, i.e. there exists a $$\gamma=\begin{pmatrix}
p&q\\
r&s
\end{pmatrix}\in\Gamma$$ such that $\gamma\sigma_1=\sigma_2.$
Since the $\sigma_k$'s are upper triangular, we find $r=0$. Having determinant 1 forces $p=s=\pm1$. Since the $\sigma_k$'s have positive entries, we find $p=s=1$. The equation $b_1\ell+qd_1=b_2\ell$ forces $q$ to be a multiple of $\ell$. Since $b$ runs over a set of representatives modulo $d$, we find $q=0$. So, the $\Gamma\sigma$'s are distinct.\par
The conditions on $f$ and $\Gamma$ then forces the coefficients of the right hand side of \eqref{product} to be modular functions for $\Gamma$. So the right hand side lies in $\mathbb C(f)[X]$. We then use the same method as \cite[Exercise 11.10]{Cox22} with a generalisation coming from the fact that $f(\tau)$ lies in $\mathbb Q(\zeta_{\ell})((e^{\frac{2\pi i\tau}{\ell}}))$ (see \cite[page 64–5]{Lang1987}) where $\zeta_{\ell}$ is a primitive $\ell$th root of unity. This fact helps show that the $f\circ\sigma$'s are distinct. Eventually, we prove \eqref{product}.\par
Let $\mathcal G$ be a fundamental domain of $\Gamma$ containing a neighbourhood of $\infty$. If the only pole of $f$ on $\mathcal G$ is at $\infty$, then the same is true for $f\circ\sigma$. Writing
$$\prod_{\sigma\in C_{\ell,N}}(X-f\circ\sigma)=\sum_{k=0}^{\psi(N)}a_k(f)X^k$$
there must exist polynomials $R_k$ such that $a_k(f)-R_k(f)$ is $0$ at $\infty$. Then $a_k(f)-R_k(f)$ would have no poles on $\mathcal G$, and so would be $0$.\end{proof}
\indent Suppose $f$ is a Hauptmodul for $\Gamma$. We outline a way of showing that $\Gamma$ acts on the right of $\{\Gamma\sigma\}_{\sigma\in C_{l,N}}$. The idea is to reduce modulo $\ell$ via the reduction map $$\pi_{\ell}:\operatorname{SL}_2(\mathbb Z)\to\operatorname{SL}_2\mleft(\frac{\mathbb Z}{l\mathbb Z}\mright)$$ and then use a computer to check that $\Gamma$ satisfies the conditions of the following proposition:\par
\begin{proposition}\label{act}
Consider the matrices $$\beta=\begin{pmatrix}
    p&q\\
    *&*.
\end{pmatrix}\in{\operatorname{SL}_2\mleft(\frac{\mathbb Z}{l\mathbb Z}\mright)}$$
such that there exist
$$\sigma\in\frac{\left\{\begin{pmatrix}
a&0\\0&d\end{pmatrix}\,\middle|\, a,d\in\left(\frac{\mathbb Z}{\ell\mathbb Z}\right)^*\right\}
}{\left(\frac{\mathbb Z}{\ell\mathbb Z}\right)^*},\quad\gamma\in\frac{\pi_{\ell}(\Gamma)}{\left(\frac{\mathbb Z}{\ell\mathbb Z}\right)^*\cap\pi_{\ell}(\Gamma)}$$
such that $\beta$ has bottom row in $({-t},{r})\left(\frac{\mathbb Z}{l\mathbb Z}\right)^*$ and $ps+qu=0$, where $$\begin{pmatrix}
    r&s\\
    t&u
\end{pmatrix}:=\sigma\gamma.$$
If every such $\beta$ lies in $\pi_{\ell}(\Gamma),$ then $\Gamma$ acts on the right of $\{\Gamma\sigma\}_{\sigma\in C_{\ell,N}}$ by multiplication for all $N$ coprime to $\ell$.
\end{proposition}
\begin{proof}
We follow \cite[Proof of Lemma 1]{Mirokov2009}. For $\sigma\in C_{\ell,N}$ and $\gamma\in\Gamma$ let $$\begin{pmatrix}
    r&s\\
    t&u
\end{pmatrix}:=\sigma\gamma.$$
Define $$\beta:=\begin{pmatrix}
    p&q\\
    \frac{-t}{\gcd(r,t)}&\frac{r}{\gcd(r,t)}
\end{pmatrix}\in\operatorname{SL}_2(\mathbb Z)$$
where $p$ and $q$ are chosen arbitrarily.
Then $$\beta\sigma\gamma=\begin{pmatrix}
    \gcd(r,t)&ps+qu\\
    0&\frac N{\gcd(r,t)}
\end{pmatrix}.$$
By multiplying $\beta$ on the left by a power of $T:=\begin{pmatrix}1&1\\0&1\end{pmatrix}$, we can assume this lies in $C_{\ell,N}$, since multiplication by matrices in $\operatorname{GL}_n(\mathbb Z)$ preserves primitiveness.\par
Any such $\beta$ will lie in $\Gamma$ if and only if $\pi_{\ell}(\beta)\in\pi_l(\Gamma)$. If this is the case, we have $(\Gamma\sigma)\gamma=\Gamma(\beta\sigma\gamma)$.
\end{proof}
\begin{definition}
The modular function $\gamma_2$ is the continuous cube root of $j$ with $\gamma_2(i)\in\mathbb R$ \cite[page 226]{Cox22}). The Weber modular function $\mathfrak f$ is given by $$\mathfrak f(\tau)=e^{\frac{-\pi i}{24}}\frac{\eta\left(\frac{\tau+1}2\right)}{\eta(\tau)}.$$
\end{definition}
Weber studied $\mathfrak f$ in \cite{Weber1891}.
\begin{proposition}
The modular function $\gamma_2$ is of level 3. The modular function $\mathfrak f$ is of level 48. Both satisfy the conditions of \autoref{product prop}.
\end{proposition}
\begin{proof} It is well known that the only poles of $\gamma_2$ and $\mathfrak f$ are at $\infty$. It remains to find congruence subgroups satisfying the conditions of \autoref{act} for which our modular functions are Hauptmoduln.\par
Chen \cite[Proposition 4.1]{chen99} shows that $\gamma_2$ is a Hauptmodul for the group $$\Gamma=\pi_3^{-1}\left\{\begin{pmatrix}
    0&2\beta\\
    \beta&0
\end{pmatrix},\begin{pmatrix}
    \alpha&0\\
    0&\alpha
\end{pmatrix},\begin{pmatrix}
    \alpha&\beta\\
    \beta&-\alpha
\end{pmatrix}\;\middle|\;\alpha,\beta\in\left(\frac{\mathbb Z}{3\mathbb Z}\right)^*\right\}.$$ We then use PARI/GP to check that $\Gamma$ satisfies the conditions of \autoref{act}.\par
Let\[
B=ST^2S^{-1}=
\begin{pmatrix}
1 & 0 \\
-2 & 1
\end{pmatrix}.
\]
Yang and Yin \cite[Theorem 4.4]{yay15} show that $\mathfrak f$ (they call it $\mathbf f_0$) is a Hauptmodul for the subgroup $TS\Phi_0^0(24)S^{-1}T^{-1}$ where $\Phi_s^0(M)=\langle\Gamma_0(2)', T^M, B^M, TB^{s+1}\rangle$. Let $$S_2:=\begin{pmatrix}
    1&-1\\
    2&-1
\end{pmatrix}=BT^{-1}.$$ Since $T$ and $B$ generate $\Gamma_0(2)$, it follows that $T$ and $S_2$ do as well. Note that $S_2^2=-I$. Since a transversal of $\Gamma_0(2)'$ in $\Gamma_0(2)$ is given by $\Gamma_0(2)^{\text{ab}}$, by the Reidemeister--Schreier method we find that $\Gamma_0(2)'$ is generated by $\{T^nS_2TS_2T^{-n-1}\}_{n\in\mathbb Z}$ and $-I$. We can then generate $\pi_{48}(\Gamma_0(2)')$ by only considering ${n\in\frac{\mathbb Z}{48\mathbb Z}}$. PARI/GP then uses 7 minutes to check \autoref{act}.\end{proof}

\section{Bounds involving \texorpdfstring{$\ln\ln N$}{lnlnN}}
We follow \cite{BrPa22}, replacing key lemmas involving $j$ and $\Delta$ with our modular functions and sufficient controlling functions. In the case of $\gamma_2,$ the controlling function we use is simply a cube root of $\Delta$, i.e. the eighth power of the Dedekind eta function $\eta$. The $\mathfrak f$ case is complicated by the fact that it has one more cusp for our controlling function $g$ to control, but a quotient of $\eta$'s still works.
\subsection{The modular function \texorpdfstring{$\gamma_2$}{g2}}
Let $\tau=x+iy$ for $x,y\in\mathbb R$. We replace \cite[Equation 12]{BrPa22} with:
\begin{lemma}\label{im}
The imaginary part of $\tau$ is bounded by $y\allowbreak\le\frac3{2\pi}\ln\left(|\gamma_2(\tau)|+9.902\right)$.
\end{lemma}
\begin{proof}
This is a rearrangement of \cite[Lemma 2.5]{Paz19}.
\end{proof}
For any function $f:\mathbb H\to\mathbb C$ and any $\sigma\in C_{\ell,N}$ let $f_\sigma:=f(\sigma(\tau))$. We replace \cite[Equation 13]{BrPa22} with:
\begin{lemma}\label{eta}
For $\tau\in\mathcal F$, we have $\begin{aligned}
-1.9<f(\tau):=\ln\max(|\eta^{8}(\tau)|,|\gamma_2(\tau)\eta^{8}(\tau)|)
<0.376.
\end{aligned}$
\end{lemma}
\begin{proof}
The lemma follows from \cite[Lemma 2.3]{BGP23}.
\end{proof}
Let $S_{\gamma_2,N}(\tau):=\sum_{\sigma\in C_{3,N}}\ln\max(1,|(\gamma_2)_\sigma|)$.
\begin{lemma}\label{c_n}
Fix $\tau\in\mathbb H$. Let
$$
\begin{aligned}
a_{\gamma_2}(\tau)&:=0.32-\frac{1}{2} \ln(|\eta^{8}(\tau)|y^2) \\
b_{\gamma_2}(\tau)&:=0.376-\ln |\eta^{8}(\tau)|.
\end{aligned}
$$
Suppose $N_0\in\mathbb R$ with $N>N_0 \geq{e^e}$. For all integers $n \geq 0$, let
\begin{align*}
c_0(\tau) &:=a_{\gamma_2}(\tau)+\ln \left(4+\frac{b_{\gamma_2}(\tau)}{\ln N_0}\right) \\
c_{n+1}(\tau) &:=a_{\gamma_2}(\tau)+\ln 2+\ln \left(1+\frac{\ln \ln N_0+c_n(\tau)}{\ln N_0}\right).
\end{align*}
Then
\begin{align*}
S_{\gamma_2,N}(\tau) \leq 2\psi(N)\left(\ln N-2 \lambda_N+\ln \ln N+c_n(\tau)\right)\label{explicit}
\end{align*}
for all $n$ up to which $c_n(\tau)\ge0$.
\end{lemma}
\begin{proof} This is proved the same way as in \cite[Proof of Theorem 1.1, only up to Lemma 3.1]{BrPa22}, though we note that \cite{BrPa22} omitted the fact that we need $c_n(\tau)\ge0$. We need this since then
the function
$$x\mapsto\frac{\ln x+c_n(\tau)}x$$
is decreasing on $[e,\infty[$, so
\begin{align*}\ln \ln N+c_n(\tau)<\left(\frac{\ln \ln N_0+c_n(\tau)}{\ln N_0}\right) \ln N.\altqedhere\end{align*}\end{proof}

The interpolation lemma \cite[Lemma 20]{BrSu10} gives—for a real number $L>1$—
\begin{align*}
h\left(\Phi_{\gamma_2,N}\right) &\leq \max _{k\in[0,\psi(N)]\cap\mathbb Z} h(\Phi_{\gamma_2,N}(X,\gamma_2(\tau_{k,N})))+\psi(N)\left(\frac{\ln L+1}{L}+3 \ln 2\right)\\
&\leq \max _{k\in[0,\psi(N)]\cap\mathbb Z} S_{\gamma_2,N}(\tau_{k,N})+\psi(N)\left(\frac{\ln L+1}{L}+4 \ln 2\right).
\numberthis\label{S_N}
\end{align*}
for any choice of $\tau_{k,N}$ such that $\gamma_2(\tau_{k,N})=L\left(1+\frac k{\psi(N)}\right).$ Let $
\Gamma:=\left\{e^{i \theta} \left\lvert\, \frac{\pi}{3} \leq \theta \leq \frac{\pi}{2}\right.\right\} \cup\{i x \mid x \in[1, \infty)\}.
$ It is well known that $j: \Gamma \to[0, \infty)$ is a bijection. Since $\gamma_2$ is a continuous cube root of $j$, it follows that $\gamma_2: \Gamma \to[0, \infty)$ is also a bijection. So we can choose $\tau_{k,N}\in\Gamma$.\par
Let $N_0=20$. Using SageMath with 2 minutes of computation we obtain $\max_{2\le_{\mathbb Z}N\le N_0}B(N)\le 3.86$ where
$$\max _{k\in[0,\psi(N)]\cap\mathbb Z} h(\Phi_{\gamma_2,N}(X,\gamma_2(\tau_{k,N})))+\psi(N)\left(\frac{\ln L+1}{L}+3 \ln 2\right)=2\psi(N)(\ln N-2\lambda_N+\ln\ln N+B(N)),$$
using the fact that $\Phi_{\gamma_2,N}(X,\gamma_2(\tau_{k,N}))=\prod_{\sigma\in C_{3,N}}(X-\gamma_2(\sigma\tau_{k,N}))$. We optimised $L$ to be $2.54$ to reduce our computed upper bound as much as possible.\par
Continuing from \eqref{S_N}, by \autoref{c_n} we get—for $N>N_0$—
$$
h( \Phi_{\gamma_2,N} ) \leq2 \psi( N ) \left( \ln N-2 \lambda_{N}+\ln \ln N+\sup _{r\in[0,1]} c_n( \tau_r)\right)+\psi( N ) \left( \frac{\ln L+1} {L}+4 \ln 2\right)$$
for any choice of $\tau_r$ such that $\{\gamma_2(\tau_r)\mid r\in[0,1]\}=[L,2L]$. Using SageMath 10.4 we plot $c_n(\tau_r)$ for $r\in[0,1]$ and get \autoref{gamma} with $n=2$ ($n=1$ resulted in a constant of $4.55$) and optimising $L$ to be $6$.

\subsection{The Weber modular function \texorpdfstring{$\mathfrak f$}{f}}
We have $\mathfrak f(\tau)=q_{48}^{-1}\prod_{n=0}^{\infty}(1+q_2^{2n+1})$.\par
\begin{lemma}\label{imf}
The imaginary part of $\tau$ is bounded by $y\allowbreak\le\frac{24}{\pi}\ln\left(|\mathfrak f(\tau)|+1.03\right)$.
\end{lemma}
\begin{proof}
The same argument as in \cite[proof of Lemma 2.5]{Paz19} shows that $\exists!y_0>0,\allowbreak\mathfrak f(iy_0)=2e^{\frac{\pi}{24}y_0}$.
Using SageMath 10.4 we find $e^{\frac{\pi}{24}y_0}\le1.03$.
\end{proof}
We have $|\mathfrak f(\tau+2)|=|\mathfrak f(\tau)|$ and $\mathfrak f\left(\frac{-1}{\tau}\right)=\mathfrak f(\tau)$ (see \cite[Corollary 12.19]{Cox22}). Thus $\mathfrak|f|$ is invariant under the theta group $\Gamma_\theta:=\left\langle T^2, S\right\rangle$
which has a fundamental domain defined by $\mathcal F_{\theta}^{\circ}:=\{\tau\in\mathbb H:|\tau|>1, -1<\Re\tau<1\}$. So we redefine $\reallywidetilde{\tau}:=\gamma\tau\in\mathcal F_{\theta}$—where $\gamma\in\Gamma_\theta$.
We will now control $\mathfrak f$ with \begin{align*}g_{\varepsilon}(\tau)&:=\frac{|\eta\left(\frac{\tau}2\right)\eta(2\tau)|^{\frac1{6-\varepsilon}+\frac16}}{|\eta(\tau)|^{\frac13}}.\end{align*} for any small rational $\varepsilon$. For $\varepsilon>0$, this will lead to $h\left(\Phi_{\mathfrak f,N}\right)<\frac{\psi(N)}{12-2\varepsilon}(\ln N-2\lambda_N+\ln\ln N+C_{\varepsilon})$.
\begin{lemma}
The ``weight" of $g_{\varepsilon}$ is given by $g_{\varepsilon}(\tilde\tau)=(c\tau+d)^{\frac1{5.99}}g_{\varepsilon}(\tau)$.
\end{lemma}
\begin{proof} We compute transformation under $S$:
\begin{align*}\eta\left(\frac{-1}{2\tau}\right)\eta\left(\frac{-2}{\tau}\right)&=\sqrt{2i\tau}\sqrt{\frac{i\tau}2}\eta(2\tau)\eta\left(\frac{\tau}2\right)=i\tau\eta(2\tau)\eta\left(\frac{\tau}2\right).\qedhere
\end{align*}
\end{proof}
Here we choose $g:=g_{0.01}$ for computation sake.
\begin{lemma} For $\tau\in\mathcal F_{\theta}$, we have $\begin{aligned}f(\tau):=\ln\max(g(\tau),|\mathfrak f(\tau)|g(\tau))
&\leq1.05.
\end{aligned}$
\end{lemma}
\begin{proof}
Since $\eta$ is holomorphic and has no zeroes on $\mathbb H$, it follows that quotients of $\eta$'s are holomorphic on $\mathbb H$. The maximum modulus principle on the Riemann surface $\mathcal F_{\theta}^\circ$ then tells us that $f$ attains its extrema on $\partial_{\mathbb P^1(\mathbb C)}\mathcal F_{\theta}$. We then use SageMath 10.4 to plot $f$ on $\partial_{\mathbb P^1(\mathbb C)}\mathcal F_{\theta}$.
\end{proof}
Let $C'_{N}:=\left\{\begin{pmatrix}a&b\\0&d\end{pmatrix}\in\mathrm{M}_2(\mathbb Z)\middle|
ad=N,0\le b<d\right\}$, i.e. $C_N$ without the $\gcd$-condition. We follow the notation of \cite[Section 1]{KiS25}. Let $h\in\mathcal M(\ell)$. \begin{lemma}For $p\nmid l$, we have $h|\mathcal T(p^r)=\omega\prod_{\sigma\in C'_{p^r}}(h\circ\sigma)$ where $\omega^2=1$.
\end{lemma}
\begin{proof}Since $h$ has leading $q$-coefficient equal to $1$, it follows that the leading $q$-coefficient of the right hand side is
$$\omega\prod_{\substack{d|N\\0\le b<d}}\exp\left(\frac {2\pi ib}d\right),$$
which can be made equal to $1$. We also have that $h|\mathcal T(p^r)$ is defined to be the right hand side multiplied by a constant in $\tau$.
\end{proof}
Let $f(\delta):=g(\delta\tau)$. Some checks show that $$C'_N=\bigcup_{k^2\mid N}kC_{\frac N{k^2}}.$$ We show that this is a disjoint union. Assume $\mathrm{gcd}(a_1,b_1,d_1)=\mathrm{gcd}(a_2,b_2,d_2)=1$ in the matrices $$C_N'\ni\begin{pmatrix}k_1a_1&k_1b_1\\0&k_1d_1\end{pmatrix}=\begin{pmatrix}k_2a_2&k_2b_2\\0&k_2d_2\end{pmatrix}.$$
Then $k_1=\gcd(k_1a_1,k_1b_1,k_1d_1)=\gcd(k_2a_2,k_2b_2,k_2d_2)=k_2$. So, the union is disjoint.\par
So $$\sum_{\sigma\in C'_N}f(\sigma)=\sum_{k^2\mid N}\sum_{\delta\in C_{\frac N{k^2}}}f(\delta).$$\par
\begin{proposition}[M\"obius inversion]
For all positive integers $N$, assume $G(N)=\sum_{k^m|N}F\left(\frac N{k^m}\right)$. Then $F(N)=\sum_{k^m|N}\mu(k)G\left(\frac N{k^m}\right)$.
\end{proposition}
\begin{proof}
Let $N=r\prod_{n=1}^up_n^m$ with $r$ free of $m$-powers and $p_n$ prime. Then
\begin{align*}
F(N)&=G(N)-\sum_{\substack{d_1^m|N\\d_1>1}}F\left(\frac N{k_1^m}\right)\\
&=G(N)-\sum_{\substack{k_1^m|N\\k_1>1}}\left(G\left(\frac N{k_1^m}\right)-\sum_{\substack{k_2^m|\frac N{k_1^m}\\k_2>1}}F\left(\frac N{k_1^mk_2^m}\right)\right)\\
&=G(N)-\sum_{\substack{k_1^m|N\\k_1>1}}G\left(\frac N{k_1^m}\right)+\sum_{\substack{k_1^mk_2^m|N\\k_1,k_2>1}}F\left(\frac N{k_1^mk_2^m}\right)\\
&\vdotswithin=\\
&=\sum_{n=0}^u(-1)^n\sum_{\substack{\prod_{s=1}^nk_s^m|N\\\forall s\in\mathbb Z\cap[1,n],k_s>1}}G\left(\frac N{\prod_{s=1}^nk_s^m}\right)\numberthis\label{sumG}
\end{align*}
by induction, with last inductive step $$F\left(\frac N{\prod_{s=1}^{n}k_s^m}\right)=G\left(\frac N{\prod_{s=1}^{n}k_s^m}\right).$$
For all $k^m|N$, the coefficient of $G\left(\frac N{k^m}\right)$ in \eqref{sumG} is equal to the coefficient of $G\left(\frac N{k}\right)$ in 
$$\sum_{n=0}^u(-1)^n\sum_{\substack{\prod_{s=1}^nk_s|N\\\forall s\in\mathbb Z\cap[1,n],k_s>1}}G\left(\frac N{\prod_{s=1}^nk_s}\right).$$
Such coefficient is equal to $\mu(k)$ since the $m=1$ case is the M\"obius  inversion formula.
\end{proof}\par
By taking logarithms, we see that if
$G(N)=\prod_{k^m|N}F\left(\frac N{k^m}\right)$, then $F(N)=\prod_{k^m|N}G^{\mu(k)}\left(\frac N{k^m}\right)$.
Let $h|F(N):=\prod_{\sigma\in C_N}(h\circ\sigma)$ be the multiplicative Hecke operator acting on $h$ (but with a gcd-condition). Let $G(N):=\prod_{\sigma\in C'_N}(h\circ\sigma)=\prod_{n^2|N}h|F\left(\frac N{n^2}\right)$. Then
\begin{align*}
h|F(N)&=\prod_{n^2|N}G^{\mu(n)}\left(\frac N{n^2}\right)=\prod_{n^2|N}\prod_{\sigma\in C'_{\frac N{n^2}}}(h\circ\sigma)^{\mu(n)}.
\end{align*}
Let $h$ be an eta-quotient with leading $q$-term $q^n$ (eta-quotients have integral leading $q$-power). Then
\begin{align*}
\omega h|F(p)&=h|\mathcal T(p)\\
&=h^{\lambda(p)}.\tag{by \cite[Theorem 1.13]{KiS25}}
\end{align*}
The leading $q$-term in the left hand side is
\begin{align*}
\prod_{\sigma\in C_p}e^{2n\pi i\sigma\tau}&=\prod_{\sigma\in C_p}e^{2n\pi i\frac{a\tau+b}d}=e^{2n\pi i\left(\sum_{b=0}^{p-1}\frac{\tau+b}p+p\tau\right)}=e^{2n\pi i\left(\tau+\frac{p(p-1)}{2p}+p\tau\right)}=(q^{p+1}e^{(p-1)\pi i})^n=q^{n(p+1)}.
\end{align*}
The leading $q$-term in the right hand side is $e^{2n\pi i\lambda(p)}=q^{n\lambda(p)}$. So $\lambda(p)=p+1$.\par
We now calculate $h|\mathcal T(p^r)$:
\begin{align*}h|\mathcal T(p^r)\mathcal T(p)&=(h|\mathcal T(p^{r+1}))(h|\mathcal T(p^{r-1}))^p\tag{by \cite[Theorem 1.7]{KiS25}}\\
h|\mathcal T(p^{r+1})&=\frac{h|\mathcal T(p^r)\mathcal T(p)}{(h|\mathcal T(p^{r-1}))^p}\\
h|\mathcal T(p^r)&=\omega_r h^{a_{r}}
\end{align*}
for some $\omega_r^2=1$ and where
$$a_{r+1}=a_r(p+1)-a_{r-1}p,\quad a_0=1,\quad a_1=p+1,$$
since if $h|\mathcal T(p^r)=\omega_rh^{a_r},$ then $h|\mathcal T(p^r)\mathcal T(p)=(\omega_rh^{a_r})|\mathcal T(p)=\omega_r^{p+1}(h|\mathcal T(p))^{a_r}=w_r^{p+1}h^{a_r(p+1)}$. This recurrence relation has solution $a_r=A+Bp^r$. Solving for $A$ and $B$ we get $$a_r=\frac{p^{r+1}-1}{p-1}.$$
\par 
For $r\ge2$ we have
\begin{align*}
h|F(p^r)&=\prod_{p^{2m}|p^r}h|\mathcal T(p^{r-2m})^{\mu\left(p^m\right)}=\frac{h|\mathcal T(p^{r})}{h|\mathcal T(p^{r-2})}=\omega h^{p^r+p^{r-1}}=\omega h^{\psi(p^r)}
\end{align*}
for some $\omega^2=1$.\par
For $\gcd(m,n)=1$ we have
\begin{align*}
h|F(mn)&=\prod_{k^2|mn}\prod_{\sigma\in C'_{\frac{mn}{k^2}}}(h\circ\sigma)^{\mu(k)}\\
&=\prod_{l^2|n}\prod_{k^2|m}\prod_{\sigma\in C'_{\frac{mn}{k^2l^2}}}(h\circ\sigma)^{\mu(kl)}\\
&=\prod_{l^2|n}\prod_{k^2|m}\left(h|\mathcal T\left(\frac m{k^2}\right)\mathcal T\left(\frac n{l^2}\right)\right)^{\mu(kl)}\\
&=h|F(m)F(n).
\end{align*}
Since $\psi$ is multiplicative, we get
\begin{proposition}
There exists an $\omega^2=1$ such that $h|F(N)=\omega h^{\psi(N)}$ for $\gcd(l,N)=1$.\end{proposition}
Since $|\mathfrak f|$ is invariant under $T^2$, our Mahler measure bound becomes
$$S_{\mathfrak f, N}(\tau)\le\sum_{\sigma\in C_{2,N}}\ln\max(1,|\mathfrak f_{\sigma}|).$$
Once we obtain a product involving our controlling function, we notice that $g$ has $\eta\left(\frac{\tau}2\right)$ which is not part of any eta-quotient. So we convert:
\begin{align*}\prod_{\sigma\in C_{2,N}}g\left(\frac{a\tau+2b}d\right)
&=\prod_{\sigma\in C_{2,N}}g\left(2\frac{a\frac{\tau}2+b}d\right)=\prod_{\sigma\in C_N}g\left(2\sigma\frac{\tau}2\right)=g^{\psi_N}(\tau).
\end{align*}
where the last step uses the fact that $g^{599\cdot6\cdot24}(2\tau)$ is the modulus of an eta-quotient.\par
Now that we have this result, we again get something similar to \cite[Lemma 3.1]{BrPa22}:
\begin{lemma}
Let
$$
\begin{aligned}
a_{\mathfrak f}(\tau)&:=14.3-11.98\ln (g(\tau))-\ln y \\
b_{\mathfrak f}(\tau)&:=1.05-\ln |g(\tau)|.
\end{aligned}
$$
Suppose $N_0\in\mathbb R$ with $N>N_0 \geq{e^e}$. For all integers $n \geq 0$, let
\begin{align*}
c_0(\tau) &:=a_{\mathfrak f}(\tau)+\ln \left(\frac1{5.99}+\frac{b_{\mathfrak f}(\tau)}{\ln N_0}\right) \\
c_{n+1}(\tau) &:=a_{\mathfrak f}(\tau)+\ln \frac1{11.98}+\ln \left(1+\frac{\ln \ln N_0+c_n(\tau)}{\ln N_0}\right).
\end{align*}
Then 
\begin{align*}
S_{\mathfrak f,N}(\tau) \leq \frac{\psi(N)}{11.98}\left(\ln N-2 \lambda_N+\ln \ln N+c_n(\tau)\right).
\end{align*}
for all $n$ up to which $c_n(\tau)\ge0$.
\end{lemma}
Let $N_0=30$. Using SageMath we use 1 minute of computation to find $$\max_{\substack{N\in[2,N_0]\cap\mathbb Z\\\gcd(N,48)=1}}B(N)\le45.7$$ where
$$\max_{\substack{k\in[0,\psi(N)]\cap\mathbb Z}} h(\Phi_{\mathfrak f,N}(X,\gamma_2(\tau_{k,N})))+\psi(N)\left(\frac{\ln L+1}{L}+3 \ln 2\right)=2\psi(N)(\ln N-2\lambda_N+\ln\ln N+B(N)),$$
optimizing $L$ to be $2^{\frac16}$.\par
Since $j=\frac{(\mathfrak f^{24}-16)^{3}}{\mathfrak f^{24}}$, we have $\mathfrak f\colon\{iy\mid y>1\}\cup\{e^{i\pi\theta}\mid \theta\in\left[\frac{\pi}3,\frac{\pi}2\right]\}\leftrightarrow\mathbb[2^{\frac16},\infty[$.
Using SageMath we plot $c_n(\tau_k)$ for $k\in[0,1]$ and get \autoref{f} with $n=2$ ($n=1$ gives us the same constant) and optimising $L$ to be $\mathfrak f(i)$.
\section{Bounds without \texorpdfstring{$\ln\ln N$}{lnlnN}}
Let $\Gamma$ be a congruence subgroup of level $l$ with a fundamental domain $\mathcal G$. Let $\reallywidetilde{\tau}$ denote $\tau$'s representative in $\mathcal G$.
\begin{lemma}\label{farey}
Let $M$ be a positive integer. Then
\[
I_M := \left[ \frac{1}{M+1}, \frac{1}{M+1}+\ell \right)
= \bigcup_{k=1}^{M} \bigcup_{\substack{h=1 \\ \gcd(h,k)=1}}^{\ell k} I_M\left( \frac{h}{k} \right)
\]
where $I_M\left( \frac{h}{k} \right)$ are intervals of the form $[\rho_1, \rho_2)$ containing $\frac{h}{k}$ such that
\[
\frac{1}{2Mk} \leq \frac{h}{k} - \rho_1 \leq \frac{1}{(M+1)k}, \qquad
\frac{1}{2Mk} \leq \rho_2 - \frac{h}{k} \leq \frac{1}{(M+1)k}.
\]
\end{lemma}
\begin{proof}
$I_M$ is $\ell$ copies of the one in \cite[Lemma 3]{Coh84}.
\end{proof}
Let $\sigma\in C_{\ell,N}$. Let $h\in[1,\ell k]\cap\mathbb Z$, $\gcd(h,k)=1$ and $\delta=\begin{pmatrix}
s & u \\
k & -h
\end{pmatrix}
\in \mathrm{SL}_2(\mathbb{Z})$ be such that $\frac{lb}{d} \in I_M\left(\frac{h}{k}\right)$. Let $\hat{\tau}_\sigma := \delta(\tau_\sigma).$ Let $s$ and $u$ be chosen in such a way (by multiplying $\delta$ by a suitable translation matrix) such that $-\tfrac{1}{2} < \mathrm{Re}(\hat{\tau}_\sigma) \leq \tfrac{1}{2}$.
\begin{lemma}\label{imt}
For
$$M=\left\lfloor\frac d{\sqrt{Ny}}\right\rfloor$$
we have
\begin{itemize}
    \item[(a)] $\operatorname{Im} \hat{\tau}_\sigma \geq \frac{1}{2}$,
    \item[(b)] $\ln \operatorname{Im} \hat{\tau}_\sigma \leq \ln \frac{d^2}{Nyk^2}$,
    \item[(c)] $\ln \operatorname{Im} \reallywidetilde{\tau_\sigma} \leq \ln \operatorname{Im} \widehat{\tau_\sigma} + \ln 4$, and
\end{itemize}
\end{lemma}
\begin{proof}
See \cite[Proof of Lemma 3.2]{BGP23} but replace $b$ with $\ell b$. 
\end{proof}
The number of such $h$ in \autoref{farey} is $l\varphi(k)$, so we first prove the following two lemmas:
\begin{lemma}\label{sum phi/k}
For any real $M\ge1$, we have $$\sum_{k=1}^{\lfloor M\rfloor}\frac{\varphi(k)}k\le\frac{6M}{\pi^2}+\ln M+1.$$
\end{lemma}
\begin{proof}
This is \cite[Lemma 2.2]{Tre15}.
\end{proof}
\begin{lemma}\label{sum phi}
For any real $M\ge 1$, we have
$$
\left|\sum_{k=1}^{\lfloor M\rfloor}\varphi(k)- \frac{3}{\pi^2} M^2\right|\le\frac{M\ln M}2 + 1.
$$ 
\end{lemma}
\begin{proof}
This is \cite[Lemma 17]{BrSu10} but with in our case $M\in\left[1,\frac32\right]$ accounted for by adding $1$ to the error term.
\end{proof}
We are now able improve \cite[Lemma 3.3]{BGP23}:
\begin{lemma}\label{bigd} Suppose $N\ge 1$ is coprime to $l$. Then
$$
\sum_{\substack{\sigma \in C_{\ell,N} \\ d\geq \sqrt{N y}}} \ln \operatorname{Im} \reallywidetilde{\tau_\sigma} \leq\left(3.97+1.39\ell\right) \psi(N).
$$
\end{lemma}
\begin{proof}We have
$$
\begin{aligned}
\sum_{\substack{\sigma \in C_N \\
d\geq \sqrt{N y}}} \ln \operatorname{Im} \reallywidetilde{\tau_\sigma} &\leq\sum_{\substack{d \mid N \\
d \geq \sqrt{N y}}} \sum_{\substack{0 \leq b<d \\
\gcd(b,r)=1}} \ln \operatorname{Im} \hat{\tau}_\sigma+\ln (4) \psi(N)\\
\sum_{\substack{d \mid N \\
d \geq \sqrt{N y}}} \sum_{\substack{0 \leq b<d \\
\gcd(b,r)=1}} \ln \operatorname{Im} \hat{\tau}_\sigma & \leq\sum_{\substack{d \mid N \\
d \geq \sqrt{N y}}} \sum_{k=1}^M\sum_{\substack{h=1\\\gcd(h,k)=1}}^{\ell k} \sum_{\substack{\frac{\ell b}{d} \in I_M\left(\frac{h}{k}\right)\\\gcd(b,r)=1}}2 \ln \frac{d}{k\sqrt{Ny}}.
\end{aligned}
$$
The number of terms in the inner sum is bounded by the number of integers $b$ with $\gcd(b,r)=1$ in $d I_M\left(\frac{h}{k}\right)$. By \autoref{farey}, the length of $I_M\left(\frac{h}{k}\right)$ is bounded above by $\frac{2}{(M+1) k}$. For an interval of length $r$, we have $\varphi(r)$ integers coprime with $r$. So \begin{align*}\#\left\{\frac{lb}d\in I_M\left(\frac hk\right):\gcd(b,r)=1\right\}&\le\varphi(r)\left\lceil\frac{2d}{l(M+1)kr}\right\rceil\\
&<\varphi(r)\left(\frac{2d}{l(M+1)kr}+1\right).
\end{align*}
So we get
\begin{gather*}\sum_{\substack{d \mid N \\
d \geq \sqrt{N y}}} \sum_{\substack{0 \leq b<d \\
\gcd(b,r)=1}} \ln \operatorname{Im} \hat{\tau}_\gamma\le
\sum_{\substack{d \mid N \\
d \geq \sqrt{N y}}} \sum_{k=1}^M \frac{4 d \varphi(r)\varphi(k)}{(M+1)kr}\ln \frac{d}{k\sqrt{Ny}}+\sum_{\substack{d \mid N \\
d \geq \sqrt{N y}}} \sum_{k=1}^M 2\ell k\varphi(r)\ln \frac{d}{k\sqrt{Ny}}.\numberthis\label{split}\end{gather*}

We deal with the first sum in the right-hand side:
$$
\begin{aligned}
& \sum_{\substack{d \mid N \\
d \geq \sqrt{N y}}} \sum_{k=1}^M \frac{4 d \varphi(r)\varphi(k)}{(M+1) kr}\ln \frac{d}{k\sqrt{Ny}}\le\sum_{\substack{d \mid N \\
d \geq \sqrt{N y}}} \sum_{k=1}^M\frac{4 d \varphi(r)\varphi(k)}{(M+1)kr}\ln \frac{M+1}{k}.
\end{aligned}
$$
Let $f(M):=\frac1{M+1}\sum_{k=1}^M\frac{\varphi(k)}k\ln\frac{M+1}k$. Then
\begin{align*}
f(M)&=\frac1{M+1}\left(\sum_{k=1}^M\frac{\varphi(k)}k\ln\Biggl(1+\frac1M\right)
+\int_1^M\frac1u\sum_{k=1}^u\frac{\varphi(k)}k\mathrm du\Biggr)\tag{Abel's summation formula for the sequence $\left(\frac{\varphi(k)}k\right)_{k=1}^\infty$}\\
&\le\frac1{M+1}\Biggl(\frac{6(M-1)}{\pi^2}+\frac{\ln^2M}2+\ln M
+\left(\frac{6M}{\pi^2}+\ln M+1\right)\ln\left(1+\frac1M\right)\Biggr)\tag{by \autoref{sum phi/k}}\\
&=:g(M).
\end{align*}
Setting $g'(x)=0$ shows $g$ is decreasing for $x\ge 6$. SageMath 10.4 shows after 2 minutes of computation that $\max_{M\in[1,10^3]\cap\mathbb Z}f(M)<0.608$. So, for all $M$ we get $f(M)\le g(10^3)<0.639.$\par
So we get 
\begin{align*}
\sum_{\substack{d \mid N \\
d \geq \sqrt{N y}}} \sum_{k=1}^M \frac{4 d \varphi(r)\varphi(k)}{(M+1) kr}\ln \frac{d}{k\sqrt{Ny}}&<\sum_{\substack{d \mid N \\d \geq \sqrt{N y}}}\frac{2.56d \varphi(r)}{r}\\
&\le 2.56\psi(N).\tag{$\frac{d\varphi(r)}r=\#\{b\in\frac{\mathbb Z}{d\mathbb Z}\mid\gcd(a,b,d)=1\}$}
\end{align*}
We now deal with the second sum in the right-hand side of \eqref{split}. By Abel's summation formula for the sequence $(\varphi(k))_{k=1}^\infty$ we have
\begin{align*}
f(M)M^2&:=\sum_{k=1}^M\varphi(k)\ln\frac{M+1}k\\
&=\ln \left(1+\frac1M\right)\sum_{k=1}^M \varphi(k)+\int_1^M\frac1u\left(\sum_{k=1}^u\varphi(k)\right)\mathrm du\\
&\le\ln \left(1+\frac1M\right)\left(\frac{3}{\pi^2} M^2 +\frac{M\ln M}2 +1\right)+\frac{3(M^2-1)}{2\pi^2}+ \frac{M (\ln(M)-1)+1}2+{\ln M}\tag{by \autoref{sum phi}}\\
&=:g(M)M^2.
\end{align*}
Setting $g'(x)=0$ shows $g$ is decreasing for $x\ge 1$. SageMath 10.4 shows after a minute of computation that $\max_{M\in[1,10^3]\cap\mathbb Z}f(M)<0.694$ and $g(10^3)<0.16$, so $f(M)<0.694$.\par
\begin{align*}
\sum_{\substack{d \mid N \\
d \geq \sqrt{N y}}} \sum_{k=1}^M 2\varphi(r)\varphi(k)\ln \frac{d}{k\sqrt{Ny}}<\sum_{\substack{d \mid N \\d \geq \sqrt{N y}}}1.39\varphi(r)M^2\le 1.39\psi(N).\altqedhere
\end{align*}
\end{proof}
\begin{remark}
We improve the constant in \cite{BGP23} by $1.8459$, so we have proved \autoref{improve}.
\end{remark}
\subsection{The modular function \texorpdfstring{$\gamma_2$}{gamma_2}}
Using SageMath 10.4 with a minute of computation time we obtain $$\max_{\substack{N\in[1,20]\cap\mathbb Z\\\gcd(N,3)=1}}B(N)\le 3.36,$$ where
$$\max _{k\in[0,\psi(N)]\cap\mathbb Z} h(\Phi_{\gamma_2,N}(X,\gamma_2(\tau_{k,N})))+\psi(N)\left(\frac{\ln L+1}{L}+3 \ln 2\right)=2\psi(N)(\ln N-2\lambda_N+B(N)).$$
optimising $L$ to be $12$. So we let $N>20$.\par
Adding the bounds in \autoref{bigd} and \cite[Lemma 3.5]{BGP23} (which still holds for $C_{l,N}$ since the proof doesn't depend on $b$), we obtain
$$
\sum_{\sigma \in C_{N}} \ln \operatorname{Im} \reallywidetilde{\tau_\sigma} <\psi(N)(5.73+\ln \operatorname{Im}\tau).
$$
So we get
$$
\begin{aligned}
S_{\gamma_2,N}(\tau) &<2 \psi(N)\left(\ln N-2 \lambda_N+0.188\right)
+2 \psi(N)(5.73+\ln \operatorname{Im}\tau)
-\psi(N) \ln \left(|\eta^{8}(\tau)|(\operatorname{Im}  \tau)^2\right).
\end{aligned}
$$
By \eqref{S_N} we get
\begin{align*}h(\Phi_{\gamma_2,N})
&< 2 \psi(N)\biggl(\ln N-2 \lambda_N+7.31+\max _{L\le \gamma_2(\tau)\le 2L}\left(\ln \operatorname{Im}\tau-\frac12 \ln \left(|\eta^{8}(\tau)|(\operatorname{Im}  \tau)^2\right)\right)+
\frac{\ln L+1}{2L}\biggr).
\end{align*}
Using SageMath 10.4 to find the constant via plot and optimising $L$ to be $12$ we get \autoref{gammanew}.

\subsection{The growth rate of \texorpdfstring{$h(\Phi_{\gamma_2,N})$}{h(F_g2N)}}
We will show that our bound in \autoref{gammanew} has the right asymptotic behaviour. The core lemma is:
\begin{lemma}For every $N \ge 1$ coprime to $3$ we have
 $S_{\gamma_2,N}(\tau) \leq 2 \ln (\psi(N)+1)+\psi(N) \ln \max \{1,|\gamma_2(\tau)|\}\allowbreak+h\left(\Phi_{\gamma_2,N}\right)$.
\end{lemma}
\begin{proof}See \cite[Proof of Lemma 4.1]{BGP23} but replace $j$ with $\gamma_2$.\end{proof}
To obtain a lower bound on $h\left(\Phi_{\gamma_2,N}\right)$, it is thus enough to bound $S_{\gamma_2,N}(\tau)$ from below. Similar to \cite{BrPa22}, we have
\begin{align*}
S_{\gamma_2,N}(\tau)&=\sum_{\sigma \in C_{N}} \ln \max \left(\left|\reallywidetilde{\eta^{8}}_{\sigma}\right|,\left|(\gamma_2)_{\sigma} \reallywidetilde{\eta^{8}}_{\sigma}\right|\right)
+2\sum_{\sigma \in C_{N}}\left(\ln \operatorname{Im} \left(\reallywidetilde{\tau_{\sigma}}\right)-\ln\operatorname{Im} \left(\tau_{\sigma}\right)\right)
-\psi(N) \ln|\eta^{8}(\tau)|\\
&>-1.9\psi(N)+2\psi(N)\left(\ln N-2 \lambda_N-\ln \operatorname{Im}  \tau\right)-\psi(N) \ln|\eta^{8}(\tau)|.
\end{align*}
So $h(\Phi_{\gamma_2,N})\in2\psi(N)(\ln N-2\lambda_N+O_{N\to\infty}(1))$.

\subsection{The Weber modular function \texorpdfstring{$\mathfrak f$}{f}}
To get a height bound for $\varepsilon=0$, we multiply our controlling function $g_0$ by a polynomial term:
\begin{lemma}\label{adjust} For $\tau\in\mathcal F_{\theta}$, we have $f(\tau):=\ln\max(g_0(\tau),|\mathfrak f(\tau)|g_0(\tau))
\leq0.04-\frac16\ln\operatorname{Im}(\gamma\tau),$ where $\gamma\tau\in\mathcal F$.
\end{lemma}
\begin{proof}
    We plot $f(\tau)+\frac16\ln\operatorname{Im}(\gamma\tau)$ using SageMath 10.9.
\end{proof}
Using SageMath 10.9 with a minute of computation we obtain $$\max_{\substack{N\in[1, 30]\cap\mathbb Z\\\gcd(N,48)=1}}B(N)\le 47.3$$ where
$$\max _{k\in[0,\psi(N)]\cap\mathbb Z} h(\Phi_{\mathfrak f,N}(X,\mathfrak f(\tau_{k,N})))+\psi(N)\left(\frac{\ln L+1}{L}+3 \ln 2\right)=\frac{\psi(N)}{12}(\ln N-2\lambda_N+B(N)),$$
optimising $L$ to be $\mathfrak f(i)$. So we let $N>30$.\par
Adding the bounds in \autoref{bigd} and \cite[Lemma 3.5 but with $C_{l,N}$ instead of $C_N$]{BGP23}, we obtain
$$
\sum_{\gamma \in C_{2,N}} \ln \operatorname{Im} \reallywidetilde{\tau_\gamma} <\psi(N)(7.12+\ln \operatorname{Im}\tau).
$$
We have
$$
\begin{aligned}
S_{\mathfrak f,N}(\tau) &<\frac1{12}\psi(N)\left(\ln N-2 \lambda_N+12\cdot0.04\right)
+\frac3{12}\psi(N)(7.12)+\frac16\psi(N)\ln(\operatorname{Im}\tau)
-\psi(N) \ln(g(\tau)).
\end{aligned}
$$
So
\begin{align*}h(\Phi_{\mathfrak f,N})&<\frac{1}{12} \psi(N)\biggl(\ln N-2 \lambda_N+55.2+\max _{L\le \mathfrak f(\tau)\le 2L}(-12 \ln \left(g(\tau))+2\ln(\operatorname{Im}\tau)\right)+12\cdot
\frac{\ln L+1}{L}\biggr).
\end{align*}
Using SageMath 10.9 to find the constant via plot and optimising $L$ to be $\mathfrak f(i)$ we get \autoref{fnew}.

\subsection{A lower bound for \texorpdfstring{$h(\Phi_{\mathfrak f,N})$}{h(Ff,N)}}
For the lower bound we use a simple control of $\mathfrak f$:
\begin{lemma}\label{delta} For $\tau\in\mathcal F_{\theta}$ we have $\begin{aligned}\ln\max(g_0(\tau),|\mathfrak f(\tau)|g_0(\tau))>-0.5.
\end{aligned}$
\end{lemma}
\noindent So $h(\Phi_{\mathfrak f,N})>\frac{\psi(N)}{12}(\ln N-2\lambda_N+C)$ for some constant $C$. So our bound in \autoref{fnew} has the right asymptotic.

\section*{Acknowledgments}
The author thanks his PhD supervisor Florian Breuer for proposing the idea of this paper, and for continued guidance and discussion. We thank Fabien Pazuki and Wadim Zudilin for some helpful discussion about various parts of the content. We thank Tim Trudgian for suggesting the references for \autoref{sum phi/k} and \autoref{sum phi}. We thank ChatGPT for some discussion about the content, proofreading a final draft, and fixing some bugs. We thank Claude Fable for proposing a way to achieve a height bound for $\varepsilon=0$, which eventually led the author to arrive at \autoref{adjust}.  

\printbibliography
\end{document}